%% file: hdp_article.tex
\documentclass[hidelinks,onefignum,onetabnum]{siamart220329}

\input{hdp_shared}

\ifpdf
\hypersetup{
  pdftitle={Noise level free regularisation of general linear inverse problems under unconstrained white noise},
  pdfauthor={T. Jahn}
}
\fi

\begin{document}

\maketitle

\begin{abstract}
In this note we solve a general statistical inverse problem under absence of knowledge of both the noise level and the noise distribution via application of the (modified) heuristic discrepancy principle.
Hereby the unbounded (non-Gaussian) noise is controlled via introducing an auxiliary discretisation dimension and choosing it in an adaptive fashion. We first show convergence for completely arbitrary compact forward operator and ground solution. Then the uncertainty of reaching the optimal convergence rate is quantified in a specific Bayesian-like environment. 
 We conclude with numerical experiments.
\end{abstract}

\begin{keywords}
statistical inverse problems, heuristic discrepancy principle, convergence, unknown noise level
\end{keywords}

\begin{MSCcodes}

\end{MSCcodes}

\section{Introduction}\label{sec1}

We aim to solve the equation 

\begin{equation*}
	Kx=y^\delta
\end{equation*}

for $K:\mathcal{X}\to \mathcal{Y}$ a compact operator between infinite-dimensional Hilbert spaces. Hereby, $y^\delta$ is a noisy measurement of the exact data $y^\dagger\in \mathcal{R}(K)$ and $x^\dagger:=K^+y^\dagger$ is the exact minimum norm solution we would like to reconstruct. We exclude the trivial case when the range of $K$ is finite-dimensional, which yields, due to compactness of $K$, that the above equation is ill-posed in the sense that the (pseudo) inverse $K^+$ is not bounded. The data is assumed to be corrupted by additive white noise $Z$ with noise level $\delta>0$ and  we write $y^\delta=y^\dagger+\delta Z$. We stress that while $y^\dagger$ is an element of the Hilbert space $\mathcal{Y}$, the white noise $Z$ and consequently $y^\delta$ are not. The measurement has to be understood in a weak sense as follows. The noise $Z$ is defined as a Hilbert-space process, i.e. a bounded linear mapping $Z:\mathcal{Y}\to L^2(\Omega,\R)$, see \cite{bissantz2007convergence}. The paradigm is that one cannot measure elements of the infinite-dimensional space $\mathcal{Y}$ directly, but only certain components of it, where each one is corrupted by random noise. More precisely, we have access to (random) measurements
	
	$$y^\delta(y):=(y^\dagger,y) + \delta Z(y)\in L^2(\Omega,\R),$$
	
	for varying $y\in\mathcal{Y}$. For convenience we write $(y^\delta,y)=y^\delta(y)=(y^\dagger,y)+\delta(Z,y)$.  A Hilbert-space-process is called white noise, when it holds that

\begin{itemize}
	\item $\E[(Z,y)] = 0$,
	\item $\E[(Z,y),(Z,y')] = (y,y')$,
	\item $\|y'\|(Z,y) \stackrel{d}{=} \|y\| (Z,y')$
\end{itemize}

for all $y,y'\in \mathcal{Y}$. We call $Z$ a Gaussian white noise (process), when $(Z,y/\|y\|)$ is standard Gaussian, however, throughout the manuscript it is not assumed that the white noise is Gaussian, if not explicitly stated. The standard way to approach the problem is through regularisation, i.e., the unbounded inverse $K^+$ is replaced with a whole family of linear bounded reconstructions and then a particular element of this family is chosen dependent on the measurement such that the exact solution $x^\dagger$ is approached for vanishing noise level $\delta\to0$, see the classic monographs from Tikhonov et al.  \cite{tikhonov1995numerical} and Engl et al. \cite{EnglHankeNeubauer:1996}, or the more recent ones from Lu \& Pereverzev \cite{lu2013regularization}, Ito \& Jin \cite{ItoJin:2015} and Hanke \cite{hanke2017taste}, to name a few for a general overview.

 In this note a method for choosing the reconstruction is presented that does not require any additional knowledge apart from $y^\delta$, neither of the noise distribution nor of the noise level $\delta$. This class of methods is called noise level free or heuristic and has been a subject of study ever since due to its importance for practical applications, see Bauer \& Lukas \cite{bauer2011comparingparameter} or Kindermann \cite{kindermann2011convergence}.
 
  For conceptual reasons we rely on spectral cut-off regularisation using explicitly the singular value decomposition of the operator $K$. This decomposition consists of orthonormal bases {of the orthogonal complement of the null space of $K$ and the closure of the range of $K$, denoted by 
  
\begin{equation*}  
  (v_j)_{j\in\N}\subset \mathcal{N}(K)^\perp:=\left\{x\in\mathcal{X}~:~(x,v)=0~\forall v\in\mathcal{X}~\mbox{with}~Kv=0\right\}\subset \mathcal{X}
  \end{equation*}
   and  
\begin{equation*}   
   (u_j)_{j\in\N}\subset \overline{\mathcal{R}(K)}:=\overline{\left\{y\in\mathcal{Y}~:~\exists x\in\mathcal{X}~\mbox{with}~ Kx=y\right\}}\subset \mathcal{Y}
\end{equation*}
   
    ($\overline{U}$ denotes the closure of the subspace $U$ in $\mathcal{Y}$)} as well as a sequence of singular values $\sigma_1\ge \sigma_2 \ge ...\searrow 0$ such that $Kv_j=\sigma_j u_j$ and $K^*u_j=\sigma_j v_j$ for all $j\in\N$. The spectral cut-off regularisation yields a family of approximations of $x^\dagger=K^+y^\dagger= \sum_{j=1}^\infty\frac{(y^\dagger,u_j)}{\sigma_j}v_j$ through 

$$x_k^\delta:=\sum_{j=1}^k \frac{(y^\delta,u_j)}{\sigma_j}v_j.$$

 Note that for many problems the singular value decomposition of the forward operator is not known and often hard to approximate numerically. { Ways to extend the following approach to more practical regularisation methods and discretisation schemes are mentioned in Section \ref{sec5}}.
	
The central objective of the paper is to determine a good value for the truncation level $k$ for unknown $\delta$. From a practical point of view the desire to solve the inverse problem without knowledge of $\delta$ is perfectly understandable, however, delicate from the mathematical side due to the famous Bakushinskii veto \cite{bakushinskii1984remarks}. This result rigorously formalises the longstanding paradigm in inverse problems, namely that those are only solvable if additional information of the size of the noise is available. Precisely, it states that in the classical deterministic setting, where, different to here, $y^\delta=y^\dagger+\delta \xi_\delta$ with $\xi_\delta\in \mathcal{Y}, \|\xi_\delta\|\le 1$, it is impossible to solve the equation without knowledge of $\delta$ for arbitrary error directions $\xi_\delta$, and this is exactly the reason why noise level free methods are called 'heuristic'. Convergence results for noise level free methods in the classical setting are therefore subject to a noise restricted analysis, see Kindermann \& Neubauer \cite{kindermann2008convergence}, i.e.,  the additional assumption that the $\xi_\delta$ are elements of some adequate subset of $\mathcal{Y}$ has to be imposed. In the stochastic setting a possible workaround for the Bakushinskii veto is to estimate the noise level using multiple data sets $y^\delta_1, y^\delta_2,...\in \mathcal{Y}$ see Harrach et al. \cite{harrach2020beyond,10.1093/imanum/drab098} and \cite{jahn2021modified} by the author.  The usage of multiple measurements for the optimal solution of an inverse problem is classic and has been considered earlier by Mair \& Ruymgaart \cite{mair1996statistical} and  by Math\'{e} \& Pereverzev \cite{mathe2006regularization, mathe2017complexity}. Different to that, in this article we present a method which only uses one measurement $y^\delta$.

 The main difference of the white noise setting (see Bissantz et al. \cite{bissantz2007convergence} and Cavalier \cite{cavalier2011inverse}) to the classical setting of deterministic bounded noise is that the measurement $y^\delta$ is no longer an element of the Hilbert space, since its variance is unbounded $\E\|y^\delta-y^\dagger\|^2 =\sum_{j=1}^\infty\E[(y^\delta-y^\dagger,u_j)^2] = \sum_{j=1}^\infty \delta^2 = \infty$. Because of that most classic methods cannot be applied in a straight forward fashion in the white noise scenario and often methods originating from statistics, as, for example, Cross validation from Wahba \cite{wahba1977practical}, empirical risk minimisation (Cavalier et al. \cite{cavalier2002oracle}) or Lepski's method \cite{lepskii1991problem} (see Math\'{e} \& Pereverzev \cite{mathe2006regularization}) are used. Still, those methods typically require knowledge of $\delta$ or do not work for general compact operators, see Lucka et al. \cite{lucka2018risk}.

Due to the above mentioned reminiscence, regarding heuristic methods for statistical inverse problems only few results are available. An exemption is Bauer \& Rei\ss{} \cite{bauer2008regularization}, where remarkably, optimal convergence of the quasi-optimality criterion, dating back to Tikhonov \cite{tikhonov1963solution}, under Gaussian white noise in a Bayesian setup is presented.
Hereby, some drawbacks are that only a rather specific setup is covered regarding the spectral properties of the forward operator, the distribution of the white noise and the source condition. We will consider a related setup in Theorem \ref{th2} below, where we quantify the probability that the optimal error bound holds. Note that for the implementation of the method from Bauer \& Rei\ss{} additional subsampling functions have to be chosen, tailored to the type of ill-posedness and source condition, which require at least some knowledge on the underlying ground truth. The method presented in this manuscript might be a bit more user-friendly and moreover applicable in more general settings. Also, the quasi-optimality criterion has been investigated in Kindermann et al. \cite{kindermann2018quasi} for stochastic noise with finite variance, which allows to perform a classical analysis. 

Another approach to handle the unboundedness of the noise, well-known from statistical inverse problems, is to smooth the data and equation. This has been done for heuristic rules under weakly bounded deterministic noise by Kindermann \& Raik \cite{kindermann2019heuristic} recently. However, presmoothing is usually not the first choice from a numerical view point since it yields a rapid increase of the degree of ill-posedness in first place.

In the recent publications \cite{jahn2021optimal, jahn2022probabilistic} by the author of this manuscript an alternative approach to handle the white noise was presented, which constitutes of using adaptively a combination of regularisation and discretisation for the discrepancy principle due to Phillips \cite{phillips1962technique} and Morozov \cite{Morozov:1966}, one of the most popular nonheuristic methods. Here we apply the very same ideas to the heuristic discrepancy principle {(see, e.g., the doctoral thesis \cite{raik2020linear}, Definition 7)}, which was originally introduced by Hanke \& Raus \cite{hanke1996general}  for the iterative Landweber method . The heuristic discrepancy principle for choosing $k$ is defined as follows

$$k^\delta_{\rm cHD}:=\arg\min_{k\in\N}\Psi(k,y^\delta) := \arg\min_{k\in\N} \frac{1}{\sigma_k}\sqrt{\sum_{j=k}^\infty(y^\delta,u_j)^2}.$$

Due to the unboundedness of the white noise the above is not applicable, since the right hand side is infinity for all $k\in\N$. We therefore introduce a discretisation parameter $m\in\N$ and define accordingly a discretised variant

\begin{equation}\label{eq00}
k^\delta_{\rm HD}(m):=\arg\min_{k\le \frac{m}{2}}\Psi_m(k,y^\delta)=\arg\min_{k\le \frac{m}{2}} \frac{1}{\sigma_k}\sqrt{\sum_{j=k}^m(y^\delta,u_j)^2}.
\end{equation}

The reason that we restrict to $k\le m/2$ for the minimisation instead of, say, $k\le m$ is to avoid random fluctuations of few components to dominate. The particular choice of $m/2$ as the upper bound is for convenience, other choices are possible as well (as long as it is taken care of that the randomness of single components is negligible).  Note that in any case only a part of the coefficients is used for the construction of the approximation, while the other part is just used for validation. This is in the spirit of a cross-validation setup, and a related approach in the context of semi-supervised learning has been analysed by Caponnetto \& Yao \cite{caponnetto2010cross}. 

In order to determine the final approximation we have to choose the discretisation level $m$. For that we follow exactly the authors' paradigm from  \cite{jahn2021optimal, jahn2022probabilistic} and set

\begin{equation}\label{eq0}
	k^\delta_{\rm HD}:=\max_{m\in\N} k^\delta_{\rm HD}(m).
\end{equation}

	 \begin{remark}\label{rem1}
	 	In a practical setup the maximisation takes place over a finite set limited by the discretisation dimension. Still, this might be intractably large and the following property  is useful for the practical maximisation. It assures that $k^\delta_{\rm HD}(m)$ is bounded by a constant (independent of $\delta$) for $m$ large enough with a probability converging to $1$ as $\delta\to0$. Precisely, for $K:=\min\left\{k\in\N~:~\sigma_k<\sigma_1\right\}$ it holds that
	 	{
	 	\begin{equation*}
	 	\mathbb{P}\left(\arg\min_{k\le \frac{m}{2}}\Psi_m(k,y^\delta)< K,~ \forall m\ge m_\delta\right)\to 1
	 	\end{equation*}
	 }
	 	as $\delta\to0$, for suitable large $m_\delta$. This assertion will be proven below {in Section \ref{sec:rem1}}. See Remark 1.1 of \cite{jahn2021modified} and Section 3 of \cite{jahn2022probabilistic} on how to use this for the practical implementation of the modified (nonheuristic) discrepancy principle. The issues  also apply here.
	 	
	 	\end{remark}
 	
We will formulate the main results in Section \ref{sec2} and prove them in Section \ref{sec3}. The paper closes with some numerical experiments in Section \ref{sec4}.

\section{Results}\label{sec2}
We formulate the main result, which shows convergence in probability of the method for general ill-posed $K$ and general (non-Gaussian) white noise. 

\begin{theorem}\label{th1}
	For any compact operator $K$ with nonclosed range and any admissible data $y^\dagger\in\mathcal{R}(K) \cup \mathcal{R}(K)^\perp$  it holds that for all $\varepsilon>0$ that
	
	\begin{equation*}
		\mathbb{P}\left(\|x_{k^\delta_{\rm HD}}^\delta - x^\dagger\|\le \varepsilon\right)\to 1
	\end{equation*}
	
	as $\delta\to0$, i.e., $x_{k^\delta_{\rm HD}}$ converges to $x^\dagger$ in probability.
\end{theorem}

We stress that Theorem \ref{th1} explains that the Bakushinskii veto  does not apply in the white noise setting in full generality since clearly no knowledge of $\delta$ is needed for the determination of the final truncation index $k^\delta_{\rm HD}$. The type of convergence in Theorem \ref{th1} is weaker then convergence in $L^2$ (called convergence in mean integrated square error in this context), which is more commonly used in statistical inverse problems. The following counter example shows that convergence of the mean integrated square error does not hold here, even if the (component) distribution of the noise is bounded.

\begin{lemma}\label{ex1}
	For $x^\dagger=0$, let $\sigma_j^2=e^{-j}$ and assume that the white noise has distribution $\mathbb{P}\left((Z,y) = 0\right)=1/2$, $\mathbb{P}\left((Z,y)=\pm \sqrt{2}\right)=1/4$. Then, it holds that 
	
	$$\E\left[\left\| x^\delta_{k^\delta_{\rm HD}}-x^\dagger\right\|^2\right] \ge \frac{1}{4\delta}\to\infty$$
	
	as $\delta\to0$, i.e., $x^\delta_{\rm HD}$ does not converge in mean integrated square error.
\end{lemma}

 It is known that the heuristic discrepancy principle does not achieve optimal convergence rates in the standard setting of general source conditions, unless an additional regularity condition (sometimes called self-similarity) is fulfilled by the unknown ground truth $x^\dagger$, see e.g. Kindermann \& Raik \cite{kindermann2019heuristic}. This is different to the performance of nonheuristic methods, as e.g., the discrepancy principle which does not need self-similarity. We illustrate the difference of the assumptions for the basic setting of a polynomially ill-posed forward operator $K$. Here, a simple instance of a self-similarity condition is when the Fourier coefficients of the unknown fulfill $(x^\dagger,v_j)^2=j^{-\eta}$ for some $\eta>1$.  Contrary, a classical source condition would only require $(x^\dagger,v_j)^2=j^{-\eta'}(\xi,v_j)^2$, where $\xi$ is an arbitrary element in $\mathcal{X}$, which is clearly a much weaker assumption, as it  {  imposes solely an upper bound upon the coefficients $(x^\dagger,v_j)^2$, while the self-similiarity condition additionally enforces a lower bound also}.
	
	 The following result employs a setting where optimal convergence is achieved for the modified nonheuristic discrepancy principle in the sense of an oracle inequality, i.e., the method attains the minimal possible error (up to a constant factor) asymptotically. Hereby, we relax the aforementioned simple self-similarity condition in that we allow for (random) deviations. The setup is borrowed from Bauer \& Rei\ss{} \cite{bauer2008regularization} and constitutes of polynomially ill-posed operators in a Bayesian-like framework. It is assumed that the components of the unknown solution $x^\dagger$ are randomly sampled, with respective decreasing variance and can be interpreted as kind of a self-similarity. For a general introduction to Bayesian inverse problems we refer to Kaipio \& Somersalo \cite{kaipio2006statistical}.

\begin{theorem}\label{th2}
	Assume that there exist $q>0$ and $\eta>1$ such that $\sigma_j^2=j^{-q}$ and $(x^\dagger,v_j) = j^{-\frac{\eta}{2}} X_j$ for all $j\in\N$, where the $(X_j)_{j\in\N}$ are either deterministic and equal to $1$ or i.i.d. standard Gaussian. Moreover, assume that $Z$ is Gaussian white noise. Then for all $0< \delta\le 1$

	\begin{equation}\label{eqth2}
	\mathbb{P}\left( \|x^\delta_{k^\delta_{\rm HD}}-x^\dagger\|\le C_{q,\eta} \min_{k\in\N}\|x^\delta_k-x^\dagger\|\right)\ge1-4 e^{-p_{q,\eta} \delta^\frac{2(1-\eta)}{(q+\eta-1)(q+\eta)}}
	\end{equation}
	
where the constants $p_{q,\eta}$ and $C_{q,\eta}$ are given in \eqref{prob} and \eqref{bound} below.  That is, up to a constant the minimal error is attained with overwhelming probability.

\end{theorem}

Theorem \ref{th1} and \ref{th2} above show that the here proposed method is reasonable and allows for regularisation without knowledge of the noise level. This will be confirmed numerically in Section \ref{sec4} below. Clearly, it would be interesting to apply the approach to other heuristic methods as, e.g., the L-curve (Hansen \cite{hansen1992analysis}) the Hanke-Raus-rule \cite{hanke1996general} or the aforementioned quasi-optimality criterion.

	


\section{Proofs}\label{sec3}

We start with a central proposition which controls the discretised measurement errors simultaneously.

\begin{proposition}\label{prop1}
	For $\varepsilon'>0$ and $\kappa\ge \left(\frac{3}{2}+\frac{3}{2\varepsilon'}\right)^2$ it holds that
	
	\begin{align*}
	&\mathbb{P}\left(\left|\sum_{j=k+1}^m(y^\delta-y^\dagger,u_j)^2-(m-k)\delta^2\right|\le \varepsilon'(m-k)\delta^2,~\forall m\ge \kappa,~k\le m/2\right)\\
	\ge\qquad&1-\frac{1}{\varepsilon'}\E\left[\left|\frac{1}{\sqrt{\kappa}}\sum_{j=1}^{\sqrt{\kappa}}\left(\delta^{-2}(y^\delta-y^\dagger,u_j)^2-1\right)\right|\right].
	\end{align*}
	
	Moreover, 
	
	\begin{align*}
	 \lim_{\kappa\to \infty}\E\left[\left|\frac{1}{\sqrt{\kappa}}\sum_{j=1}^{\sqrt{\kappa}}\left(\delta^{-2}(y^\delta-y^\dagger,u_j)^2-1\right)\right|\right] = 0.
	\end{align*}

\end{proposition}

\begin{proof}[Proof of Proposition \ref{prop1}]
We rely on Proposition 4.1 from \cite{jahn2022probabilistic}, which states that

for

$$\Omega_\kappa:=\left\{ \left|\sum_{j=1}^l(y^\delta-y^\dagger,u_j)^2-l\delta^2\right|\le \frac{\varepsilon'}{3},~\forall l\ge \sqrt{\kappa}\right\}$$

it holds that 

$$\mathbb{P}\left(\Omega_\kappa\right)\ge 1 - \frac{3}{\varepsilon'}\E\left[\left|\frac{1}{\sqrt{\kappa}}\sum_{j=1}^{\sqrt{\kappa}}\left(\delta^{-2}(y^\delta-y^\dagger,u_j)^2-1\right)\right|\right]\stackrel{(\kappa\to\infty)}{\longrightarrow} 1$$

and whose proof consists of an application of Doob's extremal theorem to the backward martingale $\left(\frac{1}{l}\sum_{j=1}^l\left((y^\delta-y^\dagger,u_j)^2\delta^{-2}-1\right)\right)_{l\in\N}$.

From $k\le m/2$ it follows that $m+2k\le 2m$ and $k\le m-k$. Thus 

$$m+k\le 2m-k=2(m-k)+k\le 2(m-k)+(m-k)=3(m-k).
$$
 We distinguish the cases $k\le \sqrt{\kappa}$ and $k>\sqrt{\kappa}$.
  In the latter case we have that

\begin{align*}
&\sum_{j=k+1}^m(y^\delta-y^\dagger,u_j)^2\omk=\sum_{j=1}^m(y^\delta-y^\dagger,u_j)^2\omk -\sum_{j=1}^k(y^\delta-y^\dagger,u_j)^2\omk\\
\le &\left(1+\frac{\varepsilon'}{3}\right)m\delta^2 - \left(1-\frac{\varepsilon'}{3}\right)k\delta^2 = (m-k)\delta^2 + \frac{\varepsilon'}{3}(m+k)\delta^2\\
\le &(1+\varepsilon')(m-k)\delta^2 
\end{align*}
	
	and similarly
	
	\begin{align*}
&\sum_{j=k+1}^m(y^\delta-y^\dagger,u_j)^2\omk=\sum_{j=1}^m(y^\delta-y^\dagger,u_j)^2\omk -\sum_{j=1}^k(y^\delta-y^\dagger,u_j)^2\omk\\
\ge &\left(1-\frac{\varepsilon'}{3}\right)m\delta^2\omk - \left(1+\frac{\varepsilon'}{3}\right)k\delta^2\omk
\ge (1-\varepsilon')(m-k)\delta^2\omk.
	\end{align*}
	
	In the former case it holds that
	
	\begin{align*}
	\sum_{j=k+1}^m(y^\delta-y^\dagger,u_j)^2\omk&\le \sum_{j=1}^m(y^\delta-y^\dagger,u_j)^2\omk\le\left(1+\frac{\varepsilon'}{3}\right)m\delta^2\omk\\
	= &\left(1+\frac{\varepsilon'}{3}\right)(m-k)\delta^2 +\frac{\varepsilon'}{3} k\delta^2\le (1+\varepsilon')(m-k)\delta^2,
	\end{align*}
	
	since $m-k\ge k$, and
	
		\begin{align*}
	&\sum_{j=k+1}^m(y^\delta-y^\dagger,u_j)^2\omk\ge \sum_{j=\sqrt{\kappa}+1}^m(y^\delta-y^\dagger,u_j)^2\omk\\
	=&\sum_{j=1}^m(y^\delta-y^\dagger,u_j)^2\omk - \sum_{j=1}^{\sqrt{\kappa}}(y^\delta-y^\dagger,u_j)^2\omk\\
	\ge &\left(1-\frac{\varepsilon'}{3}\right)m\delta^2\omk - \left(1+\frac{\varepsilon'}{3}\right)\sqrt{\kappa}\delta^2\\
	\ge&\left(1-\frac{\varepsilon'}{3}\right)(m-k)\delta^2\omk-\left(1+\frac{\varepsilon'}{3}\right)\sqrt{\kappa}\delta^2\\
    =&\left(1-\frac{\varepsilon'}{3}\right)(m-k)\delta^2\omk -\frac{2\varepsilon'}{3}(m-k)\delta^2 \frac{\left(1+\frac{\varepsilon'}{3}\right)\sqrt{\kappa}}{\frac{2\varepsilon'}{3}(m-k)}\\
	\ge&\left(1-\frac{\varepsilon'}{3}\right)(m-k)\delta^2\omk - \frac{2\varepsilon'}{3}(m-k)\delta^2 \frac{\left(1+\frac{\varepsilon'}{3}\right)}{\frac{2\varepsilon'}{3}(\sqrt{\kappa}-1)}\ge (1-\varepsilon')(m-k)\delta^2\omk,
	\end{align*}
	
	since $\kappa\ge \left(\frac{3}{2}+\frac{3}{2\varepsilon'}\right)^2$.
	
\end{proof}

\subsection{Proof of Theorem \ref{th1}}

	We first define the 'nice' events where the error behaves regularly and thus allows for perfect control of the measurement error, subject to a small parameter $\varepsilon'>0$ and a sequence $(m_\delta)_{\delta>0}$ converging to $\infty$,  which will be specified below.
	
	\begin{align}\label{om}
		\Omega_{\delta}:&=\left\{\omega \in \Omega~:~\left|\sqrt{\sum_{j=k}^m(y^\delta - y^\dagger,u_j)^2(\omega)}-\sqrt{m-k}\delta\right|\le \varepsilon' \sqrt{m-k}\delta,\right.\\\notag
		&\qquad~~ \left.~\forall m\ge m_\delta, k\le \frac{m}{2}\right\}.
	\end{align}
	
	Proposition \ref{prop1} above yields $\mathbb{P}\left(\Omega_\delta\right)\to 1$ as $\delta\to 0$ which allows us to restrict our analysis to the events $\Omega_\delta$.  For the proof we set $\varepsilon':=1/5$ and $m_\delta:=\max(k_{\delta,\varepsilon},m_{\delta,\nu})$. The auxiliary quantities $k_{\delta,\varepsilon}, m_{\delta,\nu}, \nu$ are defined in \eqref{om1}, \eqref{om2} and below. We stress that in the end all quantities depend solely on the noise level $\delta$, on the targeted upper error bound $\varepsilon$, on the singular values of $K$ and on the unknown solution $x^\dagger$ and data $y^\dagger$. Note that in fact the proof works for any $\varepsilon'<1/4$, which would change the concrete constants.

	We start by showing that the proposed rule \eqref{eq0} is well-defined. Clearly, $k^\delta_{\rm HD}(m)$ is well-defined for all $m\in\N$ and the same holds true for $k^\delta_{\rm HD}$ (on $\Omega_\delta$), as follows directly from the next proposition. In fact the proposition already  guarantees stability of the proposed method in that it upper bounds the regularisation parameter appropriately.  We define

		\begin{equation}\label{om1}
		 k_{\delta,\varepsilon}:=\max\left\{k\in\N~:~\frac{k}{\sigma_k^2} \delta^2 \le \frac{\varepsilon^2}{4}\right\}.
		 \end{equation}

	\begin{proposition}\label{prop3}
It holds that
		
		$$k^\delta_{\rm HD}\omd \le k_{\delta,\varepsilon}$$
		
		for $\delta$ small enough.
		
	\end{proposition}
	
	\begin{proof}[Proof of Proposition \ref{prop3}]
		
		
		
		
		 Note that by definition $m_\delta\ge k_{\delta,\varepsilon}$. We first observe that since $\varepsilon$ fixed

		\begin{equation}\label{eq1}
			k_{\delta,\varepsilon}\nearrow \infty
		\end{equation}
		
		as $\delta\to0$. By definition of $k_{\delta,\varepsilon}$, \eqref{eq1} and $\|x^\dagger\|<\infty$ we deduce that for all $k>k_{\delta,\varepsilon}$

		\begin{align}\label{eq1a}
			\sqrt{\sum_{j=k}^\infty (y^\dagger,u_j)^2} &\le \sigma_{k_{\delta,\varepsilon}} \sqrt{\sum_{j=k}^\infty(x^\dagger,v_j)^2} \le \sqrt{k} \delta \frac{2\sqrt{\sum_{j=k}^\infty(x^\dagger,u_j)^2}}{\varepsilon}< \frac{\sqrt{k}}{2}\delta
		\end{align}
	
		for $\delta$ sufficiently small. Now we set
		
		$$k_\varepsilon:=\min\left\{k\in\N~:~ \sqrt{\sum_{j=k}^\infty (x^\dagger,v_j)^2}\le \frac{\varepsilon}{16} \right\}.$$ 
		
		Since $k_\varepsilon$ is independent from $\delta$ and because of \eqref{eq1} and \eqref{eq1a},  for $\delta$ sufficiently small it holds that
		
		\begin{equation}\label{eq2}
			\frac{1}{4\sqrt{2}\sigma_{k_{\delta,\varepsilon}+1}}>\frac{5}{2\sigma_{k_\varepsilon}} \quad \mbox{and}\quad \sum_{j=k}^\infty(y^\dagger,u_j)^2 \le k \delta^2,
		\end{equation}
		
		for all $k>k_{\delta,\varepsilon}$.  Keep in mind that $\varepsilon'=1/5<1/4$. Then,  for all $m/2\ge k> k_{\delta,\varepsilon}$ the reverse triangle inequality yields
		
		\begin{align*}
			\Psi_m(k,y^\delta)\omd &=\frac{1}{\sigma_k}\sqrt{\sum_{j=k}^m(y^\delta,u_j)^2}\omd\\
			  &\ge \frac{1}{\sigma_k}\left(\sqrt{\sum_{j=k}^m(y^\delta-y^\dagger,u_j)^2}-\sqrt{\sum_{j=k}^m(y^\dagger,u_j)^2}\right)\omd\\
			  &> \frac{1}{\sigma_k}\left((1-\varepsilon')\sqrt{m-k}\delta -\frac{\sqrt{k}}{2}\delta\right)\omd\ge \frac{\sqrt{m}\delta}{\sqrt{2}\sigma_k}\left(1-\varepsilon'-\frac{1}{2}\right)\omd\\
			  		&\ge \frac{1}{\sigma_{k_{\delta,\varepsilon}+1}}\frac{\sqrt{m}\delta}{4\sqrt{2}}\omd,
		\end{align*}
		
		while by the triangle inequality and definition of $k_\varepsilon$
		
		\begin{align*}
			\Psi_m(k_{\varepsilon},y^\delta)&\le \frac{1}{\sigma_{k_\varepsilon}}\sqrt{\sum_{j=k_\varepsilon}^m(y^\delta,u_j)^2}\omd\le \frac{(1+\varepsilon')\sqrt{m}\delta}{\sigma_{k_\varepsilon}}+\frac{\sqrt{\sum_{j=k_\varepsilon}^m(y^\dagger,u_j)^2}}{\sigma_{k_\varepsilon}}\\
			&\le \frac{1}{\sigma_{k_\varepsilon}} \frac{5}{4}\sqrt{m}\delta + \sqrt{\sum_{j=k_\varepsilon}^\infty(x^\dagger,v_j)^2}\\
				&< \frac{1}{\sigma_{k_\varepsilon}}\frac{5}{4}\sqrt{m}\delta + \frac{\varepsilon}{16}.
		\end{align*}
		
		Furthermore, by \eqref{eq2} on the one hand
		
		\begin{align*}
			 \frac{\sqrt{m}\delta}{\sigma_{k_{\delta,\varepsilon}+1}4\sqrt{2}}> \frac{5\sqrt{m}\delta}{2\sigma_{k_\varepsilon}},
		\end{align*}
		
		whereas by definition of $k_{\delta,\varepsilon}$ on the other hand
		
		\begin{align*}
			\frac{\sqrt{m} \delta}{4\sqrt{2}} \frac{1}{\sigma_{k_{\delta,\varepsilon}+1}} &= \frac{1}{4\sqrt{2}} \sqrt{\frac{m}{k_{\delta,\varepsilon}+1}} \frac{\sqrt{k_{\delta,\varepsilon}+1} \delta}{\sigma_{k_{\delta,\varepsilon}+1}}\\
			 &\ge \frac{1}{4\sqrt{2}}\sqrt{\frac{m}{k_{\delta,\varepsilon}+1}} \frac{\varepsilon}{2}\ge \frac{\varepsilon}{8}
		\end{align*}
		
		and thus putting the preceding two estimates together yields
		
		\begin{align*}
			\Psi_m(k,y^\delta)\omd &> \frac{1}{\sigma_{k_{\delta,\varepsilon}+1}}\frac{\sqrt{m}\delta}{4\sqrt{2}}\omd\\
			&\ge 2\max\left(\frac{1}{\sigma_{k_\varepsilon}} \frac{5\sqrt{m}\delta}{4},\frac{\varepsilon}{16}\right)\omd\\
			 &\ge \left(\frac{1}{\sigma_{k_\varepsilon}}\frac{5\sqrt{m}\delta}{4}+\frac{\varepsilon}{16}\right)\omd
			 \ge \Psi_m(k_\varepsilon,y^\delta)\omd.
		\end{align*}
		
		 We conclude $k_{\rm HD}^\delta(m)\omd \le \max(k_{\varepsilon},k_{\delta,\varepsilon})$ for all $m\ge 2k_{\delta,\varepsilon}$ and thus, for $\delta$ sufficiently small,
		
		    \begin{align*}
				\max_{m\in\N}\argmin_{k\le m/2}\Psi_m(k,y^\delta)\omd&=\max\left(\max_{m\le 2k_{\delta,\varepsilon}}\argmin_{k\le m/2}\Psi_m(k,y^\delta),\max(k_{\delta,\varepsilon},k_\varepsilon)\right)\\
				&\le\max\left(k_{\delta,\varepsilon},k_\varepsilon\right)\le k_{\delta,\varepsilon},
				\end{align*}
		
	since $k_{\varepsilon}\le k_{\delta,\varepsilon}$ for $\delta$ sufficiently small, which finishes the proof of the proposition.
		
	\end{proof}
	
	The preceding proposition guarantees that we do not stop too late. In order to obtain convergence it needs to be assured that one also stops sufficiently late.  Remember that $\varepsilon$ is the upper bound we want to achieve for the overall error $\|x_{k_{\rm HD}^\delta}^\delta-x^\dagger\|$ and set
	
	\begin{equation}\label{eq2a}
				k_\varepsilon:=\min\left\{k\in\N_0~:~\sqrt{\sum_{j=k+1}^\infty(x^\dagger,v_j)^2}\le \frac{\varepsilon}{2}\right\}.
				\end{equation}

	\begin{proposition}\label{prop4}

		It holds that 
		
		\begin{equation}
		k^\delta_{\rm HD}\omd \ge k_\varepsilon\omd
		\end{equation}
		
		for $\delta$ sufficiently small.

	\end{proposition}

	\begin{proof}[Proof of Proposition \ref{prop4}]
		We set

		\begin{equation}\label{om2}
		\nu:=\frac{1}{\sigma_{1}}\sqrt{\sum_{j=k_\varepsilon}^\infty(y^\dagger,u_j)^2}.
		\end{equation}
		
		Clearly, $\nu>0$ if $\|x^\dagger\|>\varepsilon/2$. This we assume in the following, since $\|x^\dagger\|\le \varepsilon/2$ would imply $k_\varepsilon=0$, which would conclude the proof already. Further define

		$$k_{\varepsilon,\nu}:=\min\left\{k\ge k_\varepsilon+1~:~\sqrt{\sum_{j=k}^\infty(x^\dagger,v_j)^2}\le\frac{\nu}{8}\right\},$$
		
		 with $k_\varepsilon$ given in \eqref{eq2a} above. Then for
		
		\begin{equation}\label{om3}
		m_{\delta,\nu}:=\max\left\{ m\in\N~:~ \frac{\sqrt{m}\delta}{\sigma_{k_{\varepsilon,\nu}}}\le \frac{\nu}{8(1+\varepsilon')}\right\}
		\end{equation}
		
		and $\varepsilon$ fixed there clearly hold $m_{\delta,\nu} \nearrow \infty$ as $\delta\to0$. Let $\delta$ so small that $\sigma_1^{-1}\sqrt{\sum_{j=k_\varepsilon}^{m_{\delta,\nu}}(y^\dagger,u_j)^2}\ge \nu/2$ and $m_{\delta,\nu}\ge 2 k_\varepsilon$. Now, since $m_\delta\ge m_{\delta,\nu}$ and $\varepsilon'=1/5<1/4$,  for all $k\le k_\varepsilon$ we have
		
		\begin{align*}
			\Psi_{m_{\delta,\nu}}(k,y^\delta)\omd&= \frac{1}{\sigma_k}\sqrt{\sum_{j=k}^{m_{\delta,\nu}}(y^\delta,u_j)^2}\omd\ge \frac{1}{\sigma_1}\sqrt{\sum_{j=k_\varepsilon}^{m_{\delta,\nu}}(y^\delta,u_j)^2}\omd\\
			&\ge \frac{1}{\sigma_1}\sqrt{\sum_{j=k_1^\varepsilon}^{m_{\delta,\nu}}(y^\dagger,u_j)^2}\omd - \frac{1}{\sigma_1}\sqrt{\sum_{j=k_1^\varepsilon}^{m_{\delta,\nu}}(y^\delta-y^\dagger,u_j)^2}\omd\\
			&\ge \frac{\nu}{2}\omd - \frac{1+\varepsilon'}{\sigma_1}\sqrt{m_{\delta,\nu}}\delta\omd\ge \frac{3\nu}{8}\omd.
		\end{align*}

 On the contrary,
		
		\begin{align*}
			\Psi_{m_{\delta,\nu}}\left(k_{\varepsilon,\nu},y^\delta\right)\omd&\le \frac{1}{\sigma_{k_{\varepsilon,\nu}}}\sqrt{\sum_{j=k_{\varepsilon,\nu}}^{m_{\delta,\nu}}(y^\dagger,u_j)^2}\omd +\frac{1}{\sigma_{k_{\varepsilon,\nu}}}\sqrt{\sum_{j=1}^{m_{\delta,\nu}}(y^\dagger-y^\delta,u_j)^2}\omd\\
			&\le \sqrt{\sum_{j=k_{\varepsilon,\nu}}^\infty(x^\dagger,v_j)^2}\omd + (1+\varepsilon')\frac{\sqrt{m_{\delta,\nu}}}{\sigma_{k_{\varepsilon,\nu}}}\delta \omd\\
			&\le \frac{\nu}{8}\omd+\frac{\nu}{8} \omd\le \frac{\nu}{4}\omd
		\end{align*}
		
		The preceding two estimates directly imply $\Psi_{m_{\delta,\nu}}(k,y^\delta)\omd > \Psi_{m_{\delta,\nu}}(k_{\varepsilon,\nu},y^\delta)\omd$ for all $k\le k_\varepsilon$, thus since $k_\varepsilon< k_{\varepsilon,\nu}$ we deduce $k_{\rm HD}^\delta(m_{\delta,\nu})\omd > k_\varepsilon\omd$. Therefore in particular $k_{\rm HD}^\delta \omd >k_\varepsilon\omd$, and the proof is finished.
		
	\end{proof}

	We come to the main proof of Theorem \ref{th1} and apply Proposition \ref{prop3} and \ref{prop4} to the canonical error decomposition into a data propagation and an approximation error (known as the bias-variance decomposition) and obtain, for $\delta$ sufficiently small,
	
	\begin{align*}
		\|x^\delta_{k^\delta_{\rm HD}} - x^\dagger\|\omd&=\left\|\sum_{j=1}^{k^\delta_{\rm HD}}\frac{(y^\delta,u_j)}{\sigma_j} v_j - \sum_{j=1}^\infty(x^\dagger,v_j) v_j\right\|\omd\\
		&\le \sqrt{\sum_{j=1}^{k_{\rm HD}^\delta}\frac{(y^\delta-y^\dagger,u_j)^2}{\sigma_j^2}}\omd + \sqrt{\sum_{j=k^\delta_{\rm HD}+1}^\infty(x^\dagger,v_j)^2}\omd\\
		&\le \frac{(1+\varepsilon')\sqrt{k_{\delta,\varepsilon}}\delta}{\sigma_{k_{\delta,\varepsilon}}}+\sqrt{\sum_{j=k_\varepsilon+1}^\infty(x^\dagger,v_j)^2}\le \frac{\varepsilon}{2}+\frac{\varepsilon}{2}\le \varepsilon.
	\end{align*}
	
 Proposition \ref{prop1} guarantees that $\mathbb{P}(\Omega_\delta)\to1$ as $\delta\to0$ and thus finishes the proof of the theorem.

\subsection{Proof of Lemma \ref{ex1}}
	 Let $m_\delta:=\lceil\frac{-3\log(\delta)}{1-\log(2)}\rceil+1$ and define the event
	 
	 $$\Omega_\delta':=\left\{\omega\in\Omega~:~|(Z,u_{m_\delta-1})(\omega)|=\sqrt{2},~(Z,u_j)(\omega)=0~\forall m_\delta\le j \le 2m_\delta \right\}.$$
	
	By independency and the distribution of the white noise it holds that $\mathbb{P}\left(\Omega_\delta'\right)=(1/2)^{m_\delta+2}$. Since $y^\dagger=Kx^\dagger=0$ it holds that $(y^\delta,u_j)\omdp=0$ for $m_\delta\le j\le 2m_\delta$ and $|(y^\delta,u_{m_\delta-1})|\omdp=|(Z,u_{m_\delta-1})|\delta\omdp=\sqrt{2}\delta\omdp$. Thus for all $k\le m_\delta-1$ we have
	
	\begin{align*}
	\Psi_{2m_\delta}(k,y^\delta)\omdp =&\frac{1}{\sigma_k}\sqrt{\sum_{j=k}^{2m_\delta}(y^\delta,u_j)^2}\omdp\\
	&\ge \frac{1}{\sigma_{m_\delta-1}}|(y^\delta,u_{m_\delta-1})|\omdp=e^{-(m_\delta-1)} \sqrt{2}\omdp,
	\end{align*}

	while 
	
	$$\Psi_{2m_\delta}(m_\delta,y^\delta)\omdp= \frac{1}{\sigma_{m_\delta}}\sqrt{\sum_{j=m_\delta}^{2m_\delta}(y^\delta,u_j)^2}\omdp=0.$$
	
	Therefore $k^\delta_{\rm HD}\omdp \ge k^\delta_{\rm HD}(2m_\delta)\omdp \ge m_\delta \omdp$. Furthermore, for $k\ge m_\delta$,
	
	$$\|x_k^\delta-x^\dagger\|^2\omdp=\sum_{j=1}^k\frac{(y^\delta-y^\dagger,u_j)^2}{\sigma_j^2}\omdp \ge \frac{\delta^2(Z,u_{m_\delta-1})^2}{\sigma_{m_\delta-1}^2}\omdp =2 \delta^2e^{(m_\delta-1)}\omdp.$$
	
	Finally, restricting the expectation to the event $\Omega_\delta'$ yields
	
	\begin{align*}
	\E\|x^\delta_{k^\delta_{\rm HD}}-x^\dagger\|^2 &\ge \E\left[\|x^\delta_{k^\delta_{\rm HD}}-x^\dagger\|^2 \omdp\right] \ge 2\delta^2e^{(m_\delta-1)}\mathbb{P}\left(\Omega_\delta'\right)\\
	 &= 2\delta^2e^{(m_\delta-1)}\left(1/2\right)^{m_\delta+2}
	 = \frac{\delta^2}{4} e^{(m_\delta-1)(1-\log(2))} = \frac{1}{4\delta}
	\end{align*}
	 and the proof is finished.

\subsection{Proof of Theorem \ref{th2}}
	
	We will need the following technical proposition similar to Proposition \ref{prop1} above.
	
	\begin{proposition}\label{prop5}
		Let $(X_j)_{j\in\N}$ be i.i.d. standard Gaussian. Then, for $\varepsilon>0$ it holds that
		
		$$\mathbb{P}\left(\sup_{m\ge M}\left|\frac{1}{m}\sum_{j=1}^m (X_j^2-1)\right|\ge \varepsilon\right)\le 2e^{-\frac{M}{2}\left(\varepsilon-\log(1+\varepsilon)\right)}.$$
		
	\end{proposition}

\begin{proof}[Proof of Proposition \ref{prop5}]
	We apply Doob's extremal inequality to the reverse positive submartingales $\left(\exp\lambda\frac{1}{m}\sum_{j=1}^m(X_j^2-1)\right)_{m\in\N}$ and $\left(\exp\lambda\frac{-1}{m}\sum_{j=1}^m(X_j^2-1)\right)_{m\in\N}$ and obtain, for  $\lambda>0$ to be optimised 
	
	\begin{align*}
		&\mathbb{P}\left(\sup_{m\ge M}\frac{1}{m}\left|\sum_{j=1}^m(X_j^2-1)\right|>\varepsilon\right)\\
		\le &\mathbb{P}\left(\sup_{m\ge M}\frac{1}{m}\sum_{j=1}^m(X_j^2-1)>\varepsilon\right)+\mathbb{P}\left(\sup_{m\ge M}\frac{-1}{m}\sum_{j=1}^m(X_j^2-1)>\varepsilon\right)\\
		&= \mathbb{P}\left(\sup_{m\ge M}\exp\left(\lambda\frac{1}{m}\sum_{j=1}^m(X_j^2-1)\right)>\exp(\lambda\varepsilon)\right)\\
		&\qquad+\mathbb{P}\left(\sup_{m\ge M}\exp\left(-\lambda\frac{1}{m}\sum_{j=1}^m(X_j^2-1)\right)>\exp(\lambda\varepsilon)\right)\\
		 \le &\frac{\E\left[e^{\frac{\lambda}{M}\sum_{j=1}^M \left(X_j^2-1\right)}\right] +\E\left[e^{\frac{-\lambda}{M}\sum_{j=1}^M \left(X_j^2-1\right)}\right]}{e^{\lambda\varepsilon}}\\
		 &\le e^{-\lambda\varepsilon}\left( \E\left[e^{\frac{\lambda}{M} \sum_{j=1}^M X_j^2}\right] e^{-\lambda} + \E\left[e^{-\frac{\lambda}{M}\sum_{j=1}^M X_j^2}\right]e^\lambda\right)\\
		\le &e^{-\lambda\varepsilon}\left(\frac{e^{-\lambda}}{\left(1-\frac{2\lambda}{M}\right)^\frac{M}{2}}  + \frac{e^\lambda}{\left(1+\frac{2\lambda}{M}\right)^\frac{M}{2}}\right)
		\end{align*}
	
	where we used that $\sum_{j=1}^M X_j^2$ is $\chi^2$-distributed in the fifth step.  We plug in the choice 
	$\lambda_{M,\varepsilon}=\frac{M}{2}\frac{\varepsilon}{1+\varepsilon}$  and obtain

\begin{align*}
&\mathbb{P}\left(\sup_{m\ge M}\frac{1}{m}\left|\sum_{j=1}^m(X_j^2-1)\right|>\varepsilon\right)\le  e^{-\lambda_{M,\varepsilon}\varepsilon}\left(\frac{e^{-\lambda_{M,\varepsilon}}}{\left(1-\frac{2\lambda_{M,\varepsilon}}{M}\right)^\frac{M}{2}}  + \frac{e^\lambda_{M,\varepsilon}}{\left(1+\frac{2\lambda_{M,\varepsilon}}{M}\right)^\frac{M}{2}}\right)\\
&\le 2 e^{-(1+\varepsilon)\lambda_{M,\varepsilon}- \frac{M}{2}\log\left(1-\frac{2\lambda_{M,\varepsilon}}{M}\right)} \le 2 e^{-2(\varepsilon-\log(1+\varepsilon))},
\end{align*}

	 where we used that the first term dominates the second in the second step.
	
	\end{proof}
	
	We start with the proof of the Theorem and fix an auxiliary parameter $\varepsilon<1/3$. Moreover, we assume that the $(X_j)_{j\in\N}$ are i.i.d. standard Gaussian. The case where all $X_j$ equal one is easier and the small differences will be explained at the end of the proof. We first derive a lower bound for the minimal error. Define $k_\delta:=\lceil\delta^\frac{-2}{q+\eta}\rceil$; which is the index approximately balancing the bias and the variance of $\|x_k^\delta-x^\dagger\|$, as can be seen with
	
	$$\E\|x_k^\delta-x^\dagger\|^2=\sum_{j=1}^k\frac{\E\left[(y^\delta-y^\dagger,u_j)^2\right]}{\sigma_j^2} + \sum_{j=k+1}^\infty \E\left[(x^\dagger,v_j)^2\right] \asymp k^{1+q}\delta^2 + k^{1-\eta}.$$
	
	 We show that $k_\delta$ yields an error similar to the optimal  oracle choice (clearly not determinable by practical experiment) given by the $\arg\min$ of $\|x_k^\delta-x^\dagger\|$. Later we need an additional auxiliary quantity defined as $\kappa_\delta = \lceil k_\delta^\frac{\eta-1}{q+\eta-1}\rceil-1$. Obviously we have that $\kappa_\delta\le k_\delta$. Similar to the proof of Theorem \ref{th1}, we perform the analysis on the following events
	
	\begin{align}\label{th2:om}
		\Omega_\delta:&=\left\{\omega\in\Omega~:~\frac{1}{m}\left|\sum_{j=1}^m(X_j^2(\omega)-1)\right|\le \varepsilon,~ \forall m\ge c_{q,\eta,\varepsilon'}\kappa_\delta\right\}\\\notag
		&\qquad\cap \left\{\omega\in\Omega~:~\frac{1}{m}\left|\sum_{j=1}^m(y^\delta-y^\dagger,u_j)^2(\omega)-\delta^2\right|\le \varepsilon \delta^2,~\forall m\ge k_\delta\right\},
 \end{align}
	 
where { the constant $c_{q,\eta,\varepsilon'}<1$ is given in \eqref{eq3ab}  below. On $\Omega_\delta$ we have control over the randomness. Specifically, with $\varepsilon':=1-\sqrt{1-3\varepsilon}$ it holds that

\begin{align}\label{th2:om1}
&(1-\varepsilon')\sqrt{m-k}\omd\le\sqrt{\sum_{j=k+1}^{m}X_j^2}\omd\le(1+\varepsilon')\sqrt{m-k},\\\label{th2:om2}
&(1-\varepsilon')\sqrt{m-k}\delta\omd\le\sqrt{\sum_{j=k+1}^m(y^\delta-y^\dagger,u_j)^2}\omd\le(1+\varepsilon')\sqrt{m-k}\delta,\\\label{th2:om3}
&(1-\varepsilon')\sqrt{m}\delta\omd\le\sqrt{\sum_{j=1}^m(y^\delta-y^\dagger,u_j)^2}\omd\le(1+\varepsilon')\sqrt{m}\delta,
\end{align}

for all $m,k\ge \kappa_\delta$ and $m\ge 2k$. We only show exemplary the first of the above inequalities. { Note that for $\varepsilon<1/3$ and $m\ge 2k$ it holds that $(1+\varepsilon)m-(1-\varepsilon)k \le (1+3\varepsilon)(m-k)$ and $(1-\varepsilon)m-(1+\varepsilon)k\ge (1-3\varepsilon)(m-k)$}. Using this, by definition of $\Omega_\delta$ one obtains

\begin{align*}
\sqrt{\sum_{j=k+1}^{m}X_j^2\omd} &= \sqrt{\sum_{j=1}^{m}X_j^2 - \sum_{j=1}^{k}X_j^2}\omd\le \sqrt{(1+\varepsilon)m -(1-\varepsilon)k}\\
&\le\sqrt{1+3\varepsilon}\sqrt{m-k}
\end{align*}
and
\begin{align*}
\sqrt{\sum_{j=k+1}^{m}X_j^2}\omd &= \sqrt{\sum_{j=1}^{m}X_j^2 - \sum_{j=1}^{k}X_j^2}\omd\ge \sqrt{(1-\varepsilon)m -(1+\varepsilon)k}\omd\\
&\ge\sqrt{1-3\varepsilon}\sqrt{m-k}\omd,
\end{align*}

and the claim follows.

	  We start to  prove that

\begin{equation}\label{th2:e1}
 \min_{k\in\N}\|x_k^\delta-x^\dagger\|^2\omd \ge (1-\varepsilon')^2\min\left(\frac{6^{1-\eta}}{1-\eta},2^{1+q}\right)  \delta^{2\frac{\eta-1}{\eta+q}}\omd.
 \end{equation}

First, by monotonicity and positivity, for $k\le 2k_\delta$ we have

\begin{equation*}
 	\|x_k^\delta-x^\dagger\|^2 = \sum_{j=1}^k\frac{(y^\delta-y^\dagger,u_j)^2}{\sigma_j^2} + \sum_{j=k+1}^\infty(x^\dagger,v_j)^2\ge \sum_{j=2k_\delta+1}^\infty(x^\dagger,v_j)^2.
\end{equation*}

Further, using \eqref{th2:om1} we obtain

\begin{align*}
	\|x_k^\delta-x^\dagger\|^2\omd &\ge\sum_{j=2k_\delta+1}^\infty(x^\dagger,v_j)^2\omd = \sum_{j=2k_\delta+1}^\infty j^{-\eta} X_j^2\omd\\
	&= \sum_{l=2}^\infty \sum_{s=1}^{k_\delta} (lk_\delta+s)^{-\eta} X_{lk_\delta+s}^2\omd\ge \sum_{l=2}^\infty ((l+1)k_\delta)^{-\eta}\sum_{s=1}^{k_\delta} X_{lk_\delta+s}^2\omd\\
	&\ge \sum_{l=2}^\infty \left((l+1)k_\delta\right)^{-\eta} (1-\varepsilon')^2((lk_\delta+k_\delta)-(lk_\delta))\omd\\
	&= k_\delta^{1-\eta} (1-\varepsilon')^2\sum_{l=3}^\infty l^{-\eta}\omd\ge k_\delta^{1-\eta}(1-\varepsilon')^2 \int_3^{\infty} x^{-\eta}{\rm d}x\omd\\
	&= k_\delta^{1-\eta}\frac{3^{1-\eta}(1-\varepsilon')^2}{1-\eta}\omd.
	\end{align*}

Similar, for $k\ge 2k_\delta$ we have

\begin{equation*}
\|x_k^\delta-x^\dagger\|^2 = \sum_{j=1}^k\frac{(y^\delta-y^\dagger,u_j)^2}{\sigma_j^2} + \sum_{j=k+1}^\infty(x^\dagger,v_j)^2\ge \sum_{j=1}^{2k_\delta}\frac{(y^\delta-y^\dagger,u_j)^2}{\sigma_j^2}
\end{equation*}

and using \eqref{th2:om3} obtain

\begin{align*}
	\|x_k^\delta-x^\dagger\|^2\omd &\ge \sum_{j=1}^{2k_\delta}\frac{(y^\delta-y^\dagger,u_j)^2}{\sigma_j^2}\omd\ge (2k_\delta)^{q} \sum_{j=1}^{2k_\delta}(y^\delta-y^\dagger,u_j)^2\omd\\
	&\ge (2k_\delta)^q(1-\varepsilon')^2 2k_\delta\omd
	=(1-\varepsilon')^2 2^{1+q} k_\delta^{1+q}\delta^2\omd.
\end{align*}

By the choice of $k_\delta$ it holds that $k_\delta^{1-\eta} \ge 2^{1-\eta}  \delta^{2\frac{\eta-1}{\eta+q}}$ and   $k_\delta^{1+q}\delta^2\ge \delta^{2\frac{\eta-1}{\eta+q}}$. Plugging this into the above separate estimates for $k\le 2k_\delta$ and $k\ge 2k_\delta$ finally yields \eqref{th2:e1}

\begin{align*}
 \min_{k\in\N}\|x_k^\delta-x^\dagger\|^2\omd &\ge \min\left(k_\delta^{1-\eta}\frac{3^{1-\eta}(1-\varepsilon')^2}{1-\eta},(1-\varepsilon')^2 2^{1+q} k_\delta^{1+q}\delta^2\right)\omd\\
 &\ge (1-\varepsilon')^2\min\left(\frac{6^{1-\eta}}{1-\eta},2^{1+q}\right) \delta^{2\frac{\eta-1}{\eta+q}}\omd.	
\end{align*}

We move on and first show that the modified heuristic discrepancy principle does  not stop too late. Indeed, for all $m, k\ge C_{q,\eta,\varepsilon'}k_\delta$ with 
$$C_{q,\eta,\varepsilon'}:=\max\left(\left(\frac{2(1+\varepsilon')}{1-\varepsilon'}\right)^\frac{2}{q+\eta},\left(\frac{4\sqrt{3}(1+\varepsilon')}{1-\varepsilon'}\right)^\frac{2}{q}\right)$$

  and $m\ge 2k$ we have

		\begin{align*}
	\Psi_m(k,y^\delta)\omd &=\frac{1}{\sigma_k}\sqrt{\sum_{j=k}^m(y^\delta,u_j)^2}\omd\\
	&\ge \frac{1}{\sigma_k}\left(\sqrt{\sum_{j=k}^m(y^\delta-y^\dagger,u_j)^2}-\sqrt{\sum_{j=k}^m(y^\dagger,u_j)^2}\right)\omd\\
	&> \frac{1}{\sigma_k}\left((1-\varepsilon')\sqrt{m-k-1}\delta - k^\frac{-(q+\eta)}{2} \sqrt{\sum_{j=k}^m X_j^2}\right)\omd\\
	&> \frac{1}{\sigma_k}\left((1-\varepsilon')\sqrt{m-k-1}\delta - k^{-\frac{q+\eta}{2}}(1+\varepsilon')\sqrt{m-k-1}\right)\omd\\
		&\ge \frac{\sqrt{m-k-1}\delta}{\sigma_k}\left((1-\varepsilon') - C_{q,\eta,\varepsilon'}^\frac{-(q+\eta)}{2}(1+\varepsilon')k_\delta^{-\frac{q+\eta}{2}}\delta^{-1}\right)\omd\\
	&=\frac{\sqrt{m-k-1}\delta}{\sigma_k}\left(1-\varepsilon'-C_{q,\eta,\varepsilon'}^\frac{-(q+\eta)}{2}(1+\varepsilon')\right)\omd\\
	&> \frac{C_{q,\eta,\varepsilon'}^\frac{q}{2}}{\sqrt{3}}\left(1-\varepsilon'-C_{q,\eta,\varepsilon'}^\frac{-(q+\eta)}{2}(1+\varepsilon')\right)\sqrt{m k_\delta^q}\delta\omd\\
	&\ge \frac{C_{q,\eta,\varepsilon'}^\frac{q}{2}}{\sqrt{3}} \frac{1-\varepsilon'}{2}\sqrt{m k_\delta^q} \omd \ge 2(1+\varepsilon')\sqrt{m k_\delta^q} \omd,
\end{align*}

where we used that $C_{q,\eta,\varepsilon'}^\frac{-(q+\eta)}{2}(1+\varepsilon')\le \frac{1-\varepsilon'}{2}$ (by the first argument in the definition of $C_{q,\eta,\varepsilon'}$) in the second last and plugged in (the second argument of) $C_{q,\eta,\varepsilon'}$ in the last step.
On the other hand,

\begin{align*}
	&\Psi_m(k_\delta,y^\delta)\omd\\
	= &\frac{1}{\sigma_{k_\delta}}\sqrt{\sum_{j=k_\delta}^m(y^\delta,u_j)^2}\omd\\
	\le &k_\delta^\frac{q}{2}\left(\sqrt{\sum_{j=k_\delta}^m(y^\delta-y^\dagger,u_j)^2} + k_\delta^{-\frac{q+\eta}{2}} \sqrt{\sum_{j=k_\delta}^m X_j^2}\right)\omd\\
	< &k_\delta^\frac{q}{2} \left((1+\varepsilon')\sqrt{m-k_\delta-1}\delta+(1+\varepsilon')k_\delta^\frac{-(q+\eta)}{2} \sqrt{m-k_\delta-1}\right)\omd\\
	\le &2(1+\varepsilon')\sqrt{k_\delta^q m}\delta\omd< \Psi_m(k,y^\delta),
\end{align*}

where we used the preceding estimate in the last step.

}
 Therefore, 

\begin{equation}\label{eq3}
k^\delta_{\rm HD} \omd \le C_{q,\eta,\varepsilon'} k_\delta.
\end{equation}

 Next we show that we stop sufficiently late. First,

	\begin{align*}
	\Psi_{2k_\delta}(k_\delta,y^\delta)\omd&\le \frac{1}{\sigma_{k_\delta}}\left(\sqrt{\sum_{j=k_\delta}^{2k_\delta}(y^\dagger,u_j)^2} + \sqrt{\sum_{j=k_\delta}^{2k_\delta} (y^\delta-y^\dagger,u_j)^2}\right)\omd\\
	&\le \sqrt{k_\delta^q}\left((1+\varepsilon') \sqrt{k_\delta^{-(q+\eta)} (k_\delta+1)}  +(1+\varepsilon')\sqrt{k_\delta+1}\delta\right)\\
	  &\le 2\sqrt{2}(1+\varepsilon')\sqrt{k_\delta^{1+q}} \delta.
	\end{align*}

	Then, for $k\le c_{q,\eta,\varepsilon'}\kappa_\delta$ with 
	
	\begin{equation}\label{eq3ab}
	c_{q,\eta,\varepsilon'}:=\left(2^\frac{q+\eta}{2}\frac{3\sqrt{2}(1+\varepsilon')}{1-\varepsilon'}\right)^\frac{2}{1-q-\eta}<1
	\end{equation}
	
	 (note that $\eta>1$)

	\begin{align*}
		&\Psi_{2k_\delta}(k,y^\delta)\omd\ge \frac{1}{\sigma_k}\left(\sqrt{\sum_{j=k}^{2k_\delta}(y^\dagger,u_j)^2}-\sqrt{\sum_{j=1}^{2k_\delta}(y^\delta-y^\dagger,u_j)^2}\right)\omd\\
		\ge &\frac{1}{\sigma_k}\left( \sqrt{\left(2c_{q,\eta,\varepsilon'}\kappa_\delta\right)^{-(q+\eta)}\sum_{j=\lfloor c_{q,\eta,\varepsilon'}\kappa_\delta\rfloor+1}^{2\lfloor c_{q,\eta,\varepsilon'}\kappa_\delta\rfloor}X_j^2} - \sqrt{\sum_{j=1}^{2k_\delta}(y^\delta-y^\dagger,u_j)^2}\right)\omd\\
		&> \frac{1}{\sigma_k}\left(\left(\lfloor2c_{q,\eta,\varepsilon'}\kappa_\delta\rfloor\right)^\frac{-q-\eta}{2}(1-\varepsilon') \sqrt{\lfloor c_{q,\eta,\varepsilon'}\kappa_\delta\rfloor} - (1+\varepsilon') \sqrt{2k_\delta}\delta\right)\omd\\
		&\ge(1-\varepsilon')2^{-\frac{q+\eta}{2}}c_{q,\eta,\varepsilon'}^\frac{1-q-\eta}{2}(\kappa_\delta-1)^\frac{1-q-\eta}{2} - (1+\varepsilon') \sqrt{2k_\delta}\delta\omd\\
		&\ge \left(c_{q,\eta,\varepsilon'}^\frac{1-q-\eta}{2} 2^\frac{-q-\eta}{2}(1-\varepsilon') - (1+\varepsilon')\sqrt{2}\right) \sqrt{k_\delta^{1+q}}\delta\omd\\
		&= 2\sqrt{2}(1+\varepsilon')\sqrt{k_\delta^{1+q}}\delta\omd
	\end{align*}
	
	by definition of $\kappa_\delta$ and choice of $c_{q,\eta,\varepsilon'}$. On the other hand, for $c_{q,\eta,\varepsilon'} \kappa_\delta\le k\le c_{q,\eta,\varepsilon'} k_\delta$ we also have

	\begin{align*}
		\Psi_{2k_\delta}(k,y^\delta)\omd&\ge \frac{1}{\sigma_k}\left(\sqrt{\sum_{j=k}^{2k-1}(y^\dagger,u_j)^2}-\sqrt{\sum_{j=1}^{2k_\delta}(y^\delta-y^\dagger,u_j)^2}\right)\omd\\
		&> \frac{1}{\sigma_k}\left(\left(2k\right)^\frac{-q-\eta}{2} (1-\varepsilon')\sqrt{k}- (1+\varepsilon') \sqrt{2k_\delta}\delta\right)\omd\\
				&\ge 2^\frac{-q-\eta}{2} (1-\varepsilon')k^\frac{1-q-\eta}{2}- (1+\varepsilon') \sqrt{2k_\delta}\delta k^\frac{q}{2}\omd\\
		&\ge \left(c_{q,\eta,\varepsilon'}^\frac{1-q-\eta}{2} 2^{-q-\eta}(1-\varepsilon') - (1+\varepsilon')\sqrt{2}\right) \sqrt{k_\delta^{1+q}}\delta\\
		&\ge 2\sqrt{2}(1+\varepsilon')\sqrt{k_\delta^{1+q}}\delta\omd.
	\end{align*}
	
	 Putting the last three estimates together gives $\Psi_{2k_\delta}(k,y^\delta)\omd> \Psi_{2k_\delta}(k_\delta,y^\delta)\omd$ for all $k\le c_{q,\eta,\varepsilon'}k_\delta$, which yields 
	 
	 \begin{equation}\label{eq4}
	 k^\delta_{\rm HD}\omd  > c_{q,\eta,\varepsilon'} k_\delta\omd.
	 \end{equation}
	  Thus, using \eqref{eq3} and \eqref{eq4} in the data propagation and the approximation error respectively and \eqref{th2:e1} for the minimal error, it holds that
	 
	 \begin{align*}
	  &\|x_{k_{\rm HD}^\delta}^\delta-x^\dagger\|\omd\\
	  \le &\sqrt{\sum_{j=1}^{k_{\rm HD}^\delta}\frac{(y^\delta-y^\dagger,u_j)^2}{\sigma_j^2}}\omd + \sqrt{\sum_{j=k^\delta_{\rm HD}+1}^\infty(x^\dagger,v_j)^2}\omd\\
	  \le &(1+\varepsilon') \left(C_{q,\eta,\varepsilon'}k_\delta\right)^\frac{1+q}{2}\delta\omd  + \sqrt{\sum_{j=\lfloor c_{q,\eta,\varepsilon'}k_\delta\rfloor+1}^\infty j^{-\eta} X_j^2}\omd	\\
	  \le &(1+\varepsilon')\left(C_{q,\eta,\varepsilon'}k_\delta\right)^\frac{1+q}{2}\delta\omd + \sqrt{\sum_{l=1}^\infty (lc_{q,\eta,\varepsilon'}k_\delta)^{-\eta}\sum_{j=l\lfloor c_{q,\eta,\varepsilon'}k_\delta\rfloor+1}^{(l+1)\lfloor c_{q,\eta,\varepsilon'}k_\delta\rfloor} X_j^2}\omd\\
	  \le &(1+\varepsilon')  \left(C_{q,\eta,\varepsilon'}k_\delta\right)^\frac{1+q}{2}\delta\omd + \sqrt{k_\delta^{1-\eta} c_{q,\eta,\varepsilon'}^{-\eta}(1+\varepsilon')^2 \sum_{l=1}^\infty l^{-\eta}}\omd\\
	  \le &(1+\varepsilon')\left(\sqrt{C_{q,\eta,\varepsilon'^{1+q}}}\omd + c_{q,\eta,\varepsilon'}^\frac{-\eta}{2} \sqrt{\frac{\eta}{\eta-1}}\right) \delta^\frac{\eta-1}{\eta+q}\omd\\
	  \le  &(1+\varepsilon')\left(\sqrt{C_{q,\eta,\varepsilon'}^{1+q}} + c_{q,\eta,\varepsilon'}^\frac{-\eta}{2} \sqrt{\frac{\eta}{\eta-1}}\right)\sqrt{\max\left(1,(\eta-1)4^{\eta-1}\right)}  \min_{k\in\N}\|x_k^\delta-x^\dagger\|\\
	 = &C_{q,\eta,\varepsilon} \min_{k\in\N}\|x_k^\delta-x^\dagger\|
	 \end{align*}
 	
 	with 
 	
 	\begin{equation}\label{bound0}
 	C_{q,\eta,\varepsilon}:=(1+\varepsilon')\left(\sqrt{C_{q,\eta,\varepsilon'}^{1+q}} + c_{q,\eta,\varepsilon'}^\frac{-\eta}{2} \sqrt{\frac{\eta}{\eta-1}}\right)\sqrt{\max\left(1,(\eta-1)4^{\eta-1}\right)}.
 	\end{equation}
 	
 	 Finally, by Proposition \ref{prop5} the probability can be bounded as follows

 	\begin{align*}
 		\mathbb{P}\left(\Omega_\delta\right)&\ge 1 - \mathbb{P}\left(\sup_{m\ge c_{q,\eta,\varepsilon'}\kappa_\delta}\frac{1}{m}\left|\sum_{j=1}^m \left(X_j^2-1\right)\right|>\varepsilon\right)\\
 		&\qquad\qquad-\mathbb{P}\left(\sup_{m\ge k_\delta}\frac{1}{m}\left|\sum_{j=1}^m\left((y^\delta-y^\dagger,u_j)^2-\delta^2\right)\right|> \varepsilon \delta^2\right) \\
 		&\ge 1-4 e^{-\frac{c_{q,\eta,\varepsilon'}\kappa_\delta}{2}\left(\varepsilon-\log(1+\varepsilon)\right)}= 1-4 e^{-\frac{c_{q,\eta,\varepsilon'}}{2}(\varepsilon-\log(1+\varepsilon)) \delta^{\frac{2(1-\eta)}{(q+\eta-1)(q+\eta)}}}\\
 		&=1-4 e^{-p_{q,\eta,\varepsilon} \delta^\frac{2(1-\eta)}{(q+\eta-1)(q+\eta)}},
 		\end{align*}

	with 
	
	\begin{equation}\label{prob0}
	p_{q,\eta,\varepsilon}:=\frac{c_{q,\eta,\varepsilon'}\left(\varepsilon-\log(1+\varepsilon)\right)}{2}.
	\end{equation}

{	Finally, we set $\varepsilon=1/4$ in \eqref{bound0} and \eqref{prob0} and set
	
	\begin{align}\label{bound}
	  C_{q,\eta}:&=C_{q,\eta,1/4},\\\label{prob}
	  p_{q,\eta}:&=p_{q,\eta,1/4}.
	\end{align}

 Note that different choices for $\varepsilon<1/3$ could be made, where one has to keep in mind that a smaller $\varepsilon$ yields smaller $p_{q,\eta,\varepsilon}$ and $C_{q,\eta,\varepsilon}$ in \eqref{eqth2} and hence a better bound for $\|x^\delta_{k^\delta_{\rm HD}}-x^\dagger\|$, but which holds with a smaller probability. 

If $X_j=1$ the proof simplifies at some places and we only give a sketch. First, then  $\sum_{j=k+1}^m X_j^2 = m-k$ for any $k,m\in\N$ with probability $1$, and thus \eqref{th2:om1} still clearly holds. Infact, in this case it is possible to get a better bound for $\mathbb{P}(\Omega_\delta)$, since now

\begin{align*}
\mathbb{P}\left(\Omega_\delta\right)&\ge 1-0- \mathbb{P}\left(\sup_{m\ge k_\delta}\frac{1}{m}\left|\sum_{j=1}^m\left((y^\delta-y^\dagger,u_j)^2-\delta^2\right)\right|> \varepsilon \delta^2\right)\\
&\ge 1 -2 e^{-\frac{k_\delta}{2}(\varepsilon - \log(1+\varepsilon))} = 1-2 e^{-\frac{\varepsilon-\log(1+\varepsilon)}{2} \delta^\frac{-2}{q+\eta}}
\end{align*}

 }
	
{
\subsection{Proof of Remark \ref{rem1}}\label{sec:rem1}
Let $\varepsilon'>0$ be so small such that $\frac{1-3\varepsilon'}{1+2\varepsilon'}\frac{\sigma_1}{\sigma_K}>1$ and let $m_\delta$ be so large such that $\sqrt{\sum_{j=1}^\infty (y^\dagger,u_j)^2}\le \varepsilon'\sqrt{m-K}\delta$ and $(1-2\varepsilon')\sqrt{m-k}\ge (1-3\varepsilon')\sqrt{m}$ for all $m\ge m_\delta$ and $k\le m/2$. Then, on the set $\Omega_\delta$ from \eqref{om}, for all $m\ge m_\delta$ { and} $m/2\ge k\ge K$  it holds that

\begin{align*}
\Psi_m(k,y^\delta)\omd&\ge \frac{\sqrt{\sum_{j=K}^m(y^\delta-y^\dagger,u_j)^2}}{\sigma_K}\omd - \frac{\sqrt{\sum_{j=K}^m(y^\dagger,u_j)^2}}{\sigma_K}\omd\\
&\ge (1-\varepsilon')\frac{\sqrt{m-K}\delta}{\sigma_K}\omd - \varepsilon'\frac{\sqrt{m-K}\delta}{\sigma_K}\omd\\
&\ge \frac{1-3\varepsilon'}{1+2\varepsilon'}\frac{\sigma_1}{\sigma_K} (1+2\varepsilon') \frac{\sqrt{m}}{\sigma_1}\delta\omd\\
&\ge \frac{1-3\varepsilon'}{1+2\varepsilon'}\frac{\sigma_1}{\sigma_K} \Psi_m(1,y^\delta)\omd >\Psi_m(1,y^\delta)\omd,
\end{align*}

where we have used that 

\begin{align*}
\Psi_m(1,y^\delta)\omd&\le \frac{\sqrt{\sum_{j=1}^m(y^\delta-y^\dagger,u_j)^2}\omd + \sqrt{\sum_{j=1}^m(y^\dagger,u_j)^2}}{\sigma_1}\omd\\
&\le (1+2\varepsilon')\frac{\sqrt{m}\delta}{\sigma_1}\omd
\end{align*}

in the fourth step. { Consequently, $\arg\min_{k\le \frac{m}{2}} \Psi_m(k,y^\delta)\omd<K$ and } Proposition \ref{prop1} yields $\mathbb{P}\left(\Omega_\delta\right)\to 1$ as $\delta\to0$ and concludes the proof.
}

\section{Numerical experiments}\label{sec4}
In this final section we will test the introduced method numerically. For that we took four test problems from the popular MATLAB toolbox by Hansen \cite{hansen1994regularization}. These are $D\times D$ discretisations from one dimensional Fredholm integral equations, namely Phillips' test problem, an example from gravity surveying, the backwards heat equation and the determination of the second anti derivative. These problems cover varying degrees of ill-posedness and smoothness of the unknown solution. { Three different choices for the initial discretisation dimension $D\in\{2^7,2^9,2^{11}\}$ are being used in order to indicate the infinite-dimensional behaviour. For the corruption of the data we use two different marginal distributions, namely standard Gaussian white noise and a heavy-tailed and nonsymmetric generalised Pareto-distribution. Concretely in the second case we generated the data through the command {\tt gprnd($k$,$\sigma$,$\theta$)}, with shape parameter $k:=1/3$, scale parameter $\sigma:=(1-k)\sqrt{1-2k}$ and shift parameter $\theta:=-\frac{\sigma}{1-k}$. With these choices the marginal distributions are unbiased and have variance one. Note that in the second case already the third moment is unbounded. Moreover a wide range of noise levels $\delta$ are used. The particular choices of $\delta$ are made such that the signal to noise ratio (${\rm SNR}$), given through

$${\rm SNR}:=\frac{\mbox{signal}}{\mbox{noise}} = \frac{\|y^\dagger\|}{\sqrt{\E\|y^\delta-y^\dagger\|^2}} = \frac{\|y^\dagger\|}{{\sqrt{D \delta^2}}},$$

attains the following six values: { ${\rm SNR}\in\{0.01,0.01,0.1,1,10,100,1000\}$.} Note that a higher signal to noise ratio implies a smaller $\delta$. We compare the heuristic discrepancy principle to the modified discrepancy principle (with spectral cut-off) from \cite{jahn2021optimal,jahn2022probabilistic}.

 The philosophy of the classic nonheuristic discrepancy principle is that the reconstruction is determined such that it explains the measured data up to the noise. Hence it requires knowledge of $\delta$. To obtain this, one would classically determine a truncation level $k$  such that $\|Kx_k^\delta-y^\delta\|\approx \|y^\dagger-y^\delta\|$ (since $Kx_k^\delta$ is our approximation for $y^\dagger$). Because of the white noise dilemma the aforementioned norms are infinite and thus it can be applied only after discretisation, similar to the heuristic discrepancy principle. Precisely, it is determined as follows. For a fixed fudge parameter $\tau:=1.5$ (any other choice $\tau>1$ would be legit too) and discretisation level $m\le D$ we set

$$k_{\rm dp}^\delta(m):=\min\left\{0\le k\le m~:~ \sum_{j=k+1}^m(y^\delta,u_j)^2 \le \tau m \delta^2\right\},$$

and the final choice is $k^\delta_{\rm dp}:=\max_{m\in\N}k_{\rm dp}^\delta(m)$. { In order to evaluate the overall efficiency of the methods we compare with the (unattainable) optimal choice given by

\begin{equation}
k^\delta_{\rm opt}:=\arg\min_{k\le D}\|x_k^\delta-x^\dagger\|.
\end{equation}

 This choice is also known under the term oracle (see \cite{wasserman2006all}), since one could only determine its value if an oracle told one the exact $x^\dagger$ before. We present averaged (over 100 runs) relative errors, i.e., $e_*=\rm{mean}(\|x^\delta_{k^\delta_{*}}-x^\dagger\|)/\|x^\dagger\|$ for spectral cut-off and $e_*=\rm{mean}(\|x^\delta_{k^\delta_*,m^\delta_\cdot}-x^\dagger\|)/\|x^\dagger\|$ for Landweber iteration, with  $*\in\{\rm{HD},\rm{DP}, \rm{opt}\}$. The results are displayed in Table \ref{tab:phillips:g}-\ref{tab:heat:p}.

We come to the discussion of the results. First, we observe that in general the results are insensitive to a change of the size of the initial discretisation dimension $D$, which shows that the rigorous infinite-dimensional analysis in this paper is meaningful. Regarding efficiency of the method, we see that the errors of the modified heuristic discrepancy principles are fairly close to the optimum. More precisely, apart from the case with very noisy data {($\rm{SNR}=0.01$)} in almost all instances the respective errors of the modified heuristic discrepancy principle are smaller than two times the optimal error of the oracle. In the case {$\rm{SNR}=0.01$}, the reason for the significantly worse performance of the modified principles is that here the optimal truncation level is usually $k=0$ (that is, the approximation is the zero vector). Clearly, the modified choices have problems to find this truncation level correctly in this case, since here random fluctuations of the very first components have a strong impact. Still, apart from this case the methods work rather well nonasymptotically. Moreover, it is clearly visible that the modified heuristic discrepancy principle works equally well for non-Gaussian heavy-tailed noise, since the respective results are comparable to the one with Gaussian noise. However, the modified (nonheuristic) discrepancy principle performs significantly worse for non-Gaussian noise. We stress hereby that this is due to single simulations completely spoiling the results, in particular for exponentially ill-posed problems like {\tt gravity}. Note that exponentially ill-posed problems are very sensitive to the choice of the truncation level. The additional stability for exponentially ill-posed problems of the heuristic method might be surprising at first sight, but can be explained by the fact that the singular value appears in the nominator, which makes it very unlikely that the method truncates too late. Note that the nonheuristic discrepancy principle does not use the singular values explicitly and thus lacks this additional stability.  Regarding the heuristic discrepancy principle, the results are equally well for all four considered different problem setups. We only mention here without displaying the results that if one calculates the median instead of the mean then for all cases the nonheuristic and heuristic discrepancy principle yield almost the same results.

 All in all, the numerical study indicates that the modified heuristic discrepancy principle is a stable yet efficient method for applications, since it seems to give good results in different scenarios under minimal requirements.

\begin{table}[hbt!]
	\centering
	\caption{Comparison between DP and HDP + spectral cut-off for \texttt{phillips} and Gaussian noise.\label{tab:phillips:g}}
	\setlength{\tabcolsep}{4pt}
	\begin{tabular}{c|ccc|ccc|ccc|}
		\toprule
		\multicolumn{1}{c}{}&\multicolumn{3}{c}{$D=2^7$} & \multicolumn{3}{c}{$D=2^9$} & \multicolumn{3}{c}{$D=2^{11}$} \\
		\cmidrule(l){2-4} \cmidrule(l){5-7} \cmidrule(l){8-10} 
		${\rm SNR}$ & $e_{\rm HDP}$  & $e_{\rm DP}$ & $e_{\rm opt}$ & $e_{\rm HDP}$  & $e_{\rm DP}$ & $e_{\rm opt}$ & $e_{\rm HDP}$  & $e_{\rm DP}$ & $e_{\rm opt}$\\
		1e-2 & 1.1e1  &  5.3e0    & 9.8e-1   &  6.1e0    &  2.8e0    &  9.6e-1 & 2.8e0& 1.4e0& 9.4e-1\\
		1e-1 & 1.4e0   & 1.3e0     & 8.3e-1   & 9.1e-1    &  8.8e-1    & 7e-1   & 6.2e-1& 6e-1& 4.7e-1\\
		1e0 & 4e-1     & 3.7e-1      & 2.7e-1  & 3e-1   &  2.5e-1    & 1.7e-1 & 1.7e-1& 1.3e-1& 1.2e-1\\
		1e1 & 1.2e-1     & 1.1e-1      & 6.1e-2  & 1.1e-1   &  6.5e-2    & 4e-2 & 9.6e-2& 3e-2& 2.9e-2\\
		1e2  & 2.9e-2     & 2.6e-2      & 2.4e-2  & 2.6e-2   &  2.5e-2    & 1.9e-2 & 2.5e-2&2.5e-2 & 1.5e-2\\
		1e3 & 2.5e-2     & 1.5e-2      & 1.0e-2  & 2.4e-2   &  1.2e-2    & 6.7e-3 & 2.3e-2&8.6e-3 &5.3e-3  \\								
		\bottomrule
	\end{tabular}
\end{table}

\begin{table}[hbt!]
	\centering
	\caption{Comparison between DP and HDP + spectral cut-off for \texttt{phillips} and heavy-tailed noise.\label{tab:phillips:p}}
	\setlength{\tabcolsep}{4pt}
	\begin{tabular}{c|ccc|ccc|ccc|}
		\toprule
		\multicolumn{1}{c}{}&\multicolumn{3}{c}{$D=2^7$} & \multicolumn{3}{c}{$D=2^9$} & \multicolumn{3}{c}{$D=2^{11}$} \\
		\cmidrule(l){2-4} \cmidrule(l){5-7} \cmidrule(l){8-10} 
		${\rm SNR}$ & $e_{\rm HDP}$  & $e_{\rm DP}$ & $e_{\rm opt}$ & $e_{\rm HDP}$  & $e_{\rm DP}$ & $e_{\rm opt}$ & $e_{\rm HDP}$  & $e_{\rm DP}$ & $e_{\rm opt}$\\
		1e-2 & 1e1  &  3.1e4    & 9.9e-1   &  5.5e0    &  6.3e3    &  9.7e-1 & 2.8e0& 1.7e0& 9.5e-1\\
		1e-1 & 1.3e0   & 1.2e3     & 8.3e-1   & 8.9e-1    &  4.3e3    & 6.6e-1   & 6.3e-1& 8.1e0& 4.7e-1\\
		1e0 & 3.9e-1     & 2.1e2      & 2.6e-1  & 2.9e-1   &  3.5e0    & 1.5e-1 & 1.6e-1& 1.6e0& 1.1e-1\\
		1e1 & 1.1e-1     & 2.4e1      & 5.9e-2  & 1.1e-1   &  4.1e-1    & 3.9e-2 & 9.5e-2& 8.7e-2& 2.9e-2\\
		1e2  & 3.5e-2     & 1e1      & 2.4e-2  & 2.6e-2   &  2.6e-2    & 1.9e-2 & 2.5e-2&2.9e-1 & 1.5e-2\\
		1e3 & 2.5e-2     & 2.6e-1      & 9.3e-3  & 2.4e-2   &  1.3e-2    & 7.2e-3 & 2.1e-2&8.6e-3 &5.1e-3  \\								
		\bottomrule
	\end{tabular}
\end{table}

\begin{table}[hbt!]
	\centering
	\caption{Comparison between DP and HDP + spectral cut-off for \texttt{deriv2} and Gaussian noise.\label{tab:deriv2:g}}
	\setlength{\tabcolsep}{4pt}
	\begin{tabular}{c|ccc|ccc|ccc|}
		\toprule
		\multicolumn{1}{c}{}&\multicolumn{3}{c}{$D=2^7$} & \multicolumn{3}{c}{$D=2^9$} & \multicolumn{3}{c}{$D=2^{11}$} \\
		\cmidrule(l){2-4} \cmidrule(l){5-7} \cmidrule(l){8-10} 
		${\rm SNR}$ & $e_{\rm HDP}$  & $e_{\rm DP}$ & $e_{\rm opt}$ & $e_{\rm HDP}$  & $e_{\rm DP}$ & $e_{\rm opt}$ & $e_{\rm HDP}$  & $e_{\rm DP}$ & $e_{\rm opt}$\\
		1e-2 & 6.2e0  &  1.3e1    & 9.7e-1   &  3e0    &  4.4e0    &  9.4e-1 & 1.5e0& 3.2e0& 8.8e-1\\
		1e-1 & 8.8e-1   & 2.9e0     & 7.9e-1   & 6.7e-1    &  1.5e0    & 6.2e-1   & 6e-1& 7.3e-1& 5.5e-1\\
		1e0 & 5.6e-1     & 5.8e-1      & 5e-1  & 5e-1   &  5.4e-1    & 4.5e-1 & 4.8e-1& 4.8e-1& 4e-1\\
		1e1 & 4.3e-1     & 3.9e-1      & 3.5e-1  & 3.7e-1   &  3.7e-1    & 3.1e-1 & 3.4e-1& 3.2e-1& 2.8e-1\\
		1e2  & 2.8e-1     & 2.8e-1      & 2.4e-1  & 2.6e-1   &  2.5e-1    & 2.2e-1 & 2.3e-1&2.3e-1 & 1.9e-1\\
		1e3 & 2e-1     & 1.9e-1      & 1.7e-1  & 1.8e-1   &  1.7e-1    & 1.5e-1 & 1.6e-2&1.6e-1 &1.3e-1  \\								
		\bottomrule
	\end{tabular}
\end{table}

\begin{table}[hbt!]
	\centering
	\caption{Comparison between DP and HDP + spectral cut-off for \texttt{deriv2} and heavy-tailed noise.\label{tab:deriv2:p}}
	\setlength{\tabcolsep}{4pt}
	\begin{tabular}{c|ccc|ccc|ccc|}
		\toprule
		\multicolumn{1}{c}{}&\multicolumn{3}{c}{$D=2^7$} & \multicolumn{3}{c}{$D=2^9$} & \multicolumn{3}{c}{$D=2^{11}$} \\
		\cmidrule(l){2-4} \cmidrule(l){5-7} \cmidrule(l){8-10} 
		${\rm SNR}$ & $e_{\rm HDP}$  & $e_{\rm DP}$ & $e_{\rm opt}$ & $e_{\rm HDP}$  & $e_{\rm DP}$ & $e_{\rm opt}$ & $e_{\rm HDP}$  & $e_{\rm DP}$ & $e_{\rm opt}$\\
		1e-2 & 5.8e0  &  6.8e2    & 9.7e-1   &  3.6e0    &  2.4e3    &  9.8e-1 & 1.6e0& 1.4e0& 8.9e-1\\
		1e-1 & 8.5e-1   & 2.1e2     & 7.5e-1   & 6.9e-1    &  3e2    & 6.3e-1   & 6.5e-1& 1.1e1& 5.5e-1\\
		1e0 & 5.4e-1     & 9.3e1      & 4.9e-1  & 5e-1   &  2.1e0    & 4.5e-1 & 4.7e-1& 1.4e0& 3.9e-1\\
		1e1 & 4.1e-1     & 1.7e1      & 3.4e-1  & 3.7e-1   &  1.7e0    & 3.1e-1 & 3.3e-1& 6.7e-1& 2.8e-1\\
		1e2  & 2.8e-1     & 7e-1      & 2.3e-1  & 2.5e-1   &  2.7e-1    & 2.2e-1 & 2.3e-1&2.3e-1 & 1.9e-1\\
		1e3 & 1.9e-1     & 3.5e-1      & 1.6e-1  & 1.7e-1   &  1.9e-1    & 1.5e-1 & 1.6e-2&1.6e-1 &1.3e-1  \\								
		\bottomrule
	\end{tabular}
\end{table}

\begin{table}[hbt!]
	\centering
	\caption{Comparison between DP and HDP + spectral cut-off for \texttt{gravity} and Gaussian noise.\label{tab:gravity:g}}
	\setlength{\tabcolsep}{4pt}
	\begin{tabular}{c|ccc|ccc|ccc|}
		\toprule
		\multicolumn{1}{c}{}&\multicolumn{3}{c}{$D=2^7$} & \multicolumn{3}{c}{$D=2^9$} & \multicolumn{3}{c}{$D=2^{11}$} \\
		\cmidrule(l){2-4} \cmidrule(l){5-7} \cmidrule(l){8-10} 
		${\rm SNR}$ & $e_{\rm HDP}$  & $e_{\rm DP}$ & $e_{\rm opt}$ & $e_{\rm HDP}$  & $e_{\rm DP}$ & $e_{\rm opt}$ & $e_{\rm HDP}$  & $e_{\rm DP}$ & $e_{\rm opt}$\\
		1e-2 & 6.4e0  &  4e0    & 9.7e-1   &  3.7e0    &  5.8e0    &  9.5e-1 & 2e0& 2.7e0& 9.1e-1\\
		1e-1 & 9.3e-1   & 1.5e0     & 7.2e-1   & 6.9e-1    &  7.1e-1    & 5.5e-1   & 4.5e-1& 6e-1& 3.9e-1\\
		1e0 & 2.9e-1     & 3.5e-1      & 2.5e-1  & 2.1e-1   &  2.5e-1    & 1.9e-1 & 1.5e-1& 1.8e-1& 1.3e-1\\
		1e1 & 9.5e-2     & 1.1e-1      & 8.6e-2  & 7.4e-2   &  9.1e-2    & 6.8e-2 & 5.8e-2& 6.9e-2& 4.9e-2\\
		1e2  & 4.3e-2     & 5.3e-2      & 3.5e-2  & 3.3e-2   &  4e-2    & 2.8e-2 & 2.8e-2&3.1e-2 & 2.1e-2\\
		1e3 & 2e-2     & 2.3e-2      & 1.6e-2  & 1.5e-2   &  2e-2    & 1.2e-2 & 1.2e-2&1.4e-2 &9.4e-3  \\								
		\bottomrule
	\end{tabular}
\end{table}

\begin{table}[hbt!]
	\centering
	\caption{Comparison between DP and HDP + spectral cut-off for \texttt{gravity} and heavy-tailed noise.\label{tab:gravity:p}}
	\setlength{\tabcolsep}{4pt}
	\begin{tabular}{c|ccc|ccc|ccc|}
		\toprule
		\multicolumn{1}{c}{}&\multicolumn{3}{c}{$D=2^7$} & \multicolumn{3}{c}{$D=2^9$} & \multicolumn{3}{c}{$D=2^{11}$} \\
		\cmidrule(l){2-4} \cmidrule(l){5-7} \cmidrule(l){8-10} 
		${\rm SNR}$ & $e_{\rm HDP}$  & $e_{\rm DP}$ & $e_{\rm opt}$ & $e_{\rm HDP}$  & $e_{\rm DP}$ & $e_{\rm opt}$ & $e_{\rm HDP}$  & $e_{\rm DP}$ & $e_{\rm opt}$\\
		1e-2 & 6.2e0  &  6.2e7    & 9.5e-1   &  3.4e0    &  1.4e16    &  9.5e-1 & 2e0& 4.3e5& 8.9e-1\\
		1e-1 & 1.1e0   & 2.4e15     & 7.3e-1   & 6.8e-1    &  1.8e6    & 5.4e-1   & 4.5e-1& 1.2e14& 3.9e-1\\
		1e0 & 2.8e-1     & 1.9e14      & 2.5e-1  & 2e-1   &  2.9e14    & 1.7e-1 & 1.5e-1& 19.2e6& 1.3e-1\\
		1e1 & 9.1e-2     & 2.9e-1      & 8.2e-2  & 7.2e-2   &  5.8e4    & 6.5e-2 & 5.9e-2& 7.3e-2& 4.9e-2\\
		1e2  & 3.9e-2     & 2.7e10      & 3.3e-2  & 3.2e-2   &  1.1e11    & 2.7e-2 & 2.6e-2&3.6e11 & 2.1e-2\\
		1e3 & 1.8e-2     & 1.8e7      & 1.5e-2  & 1.5e-2   &  2.9e10    & 1.2e-2 & 1.2e-2&7.3e-2 &9.4e-3  \\								
		\bottomrule
	\end{tabular}
\end{table}

\begin{table}[hbt!]
	\centering
	\caption{Comparison between DP and HDP + spectral cut-off for \texttt{heat} and Gaussian noise.\label{tab:heat:g}}
	\setlength{\tabcolsep}{4pt}
	\begin{tabular}{c|ccc|ccc|ccc|}
		\toprule
		\multicolumn{1}{c}{}&\multicolumn{3}{c}{$D=2^7$} & \multicolumn{3}{c}{$D=2^9$} & \multicolumn{3}{c}{$D=2^{11}$} \\
		\cmidrule(l){2-4} \cmidrule(l){5-7} \cmidrule(l){8-10} 
		${\rm SNR}$ & $e_{\rm HDP}$  & $e_{\rm DP}$ & $e_{\rm opt}$ & $e_{\rm HDP}$  & $e_{\rm DP}$ & $e_{\rm opt}$ & $e_{\rm HDP}$  & $e_{\rm DP}$ & $e_{\rm opt}$\\
		1e-2 & 4.9e0  &  6.7e0    & 9.9e-1   &  2.4e0    &  2.5e0    &  1e0 & 1.4e0& 1.4e0& 9.7e-1\\
		1e-1 & 1.1e0   & 1.2e0     & 9.1e-1   & 9.4e-1    &  1e0    & 8.5e-1   & 8.3e-1& 8.7e-1& 7.6e-1\\
		1e0 & 7e-1     & 7.3e-1      & 6.7e-1  & 6.8e-1   &  6.8e-1    & 5.5e-1 & 6.6e-1& 6e-1& 3.9e-1\\
		1e1 & 5.4e-1     & 3.9e-1      & 2.8e-1  & 2.9e-1   &  3e-1    & 2.2e-1 & 2.6e-1& 2.6e-1& 1.5e-1\\
		1e2  & 1.7e-1     & 1.6e-1      & 1.1e-1  & 1e-1   &  1.1e-1    & 7.7e-2 & 9.4e-2&9.6e-2 & 5.6e-2\\
		1e3 & 4.5e-2     & 5.7e-2      & 3.6e-2  & 4.4e-2   &  4.6e-2    & 2.5e-2 & 2.2e-2&3e-2 &2e-2  \\								
		\bottomrule
	\end{tabular}
\end{table}

\begin{table}[hbt!]
	\centering
	\caption{Comparison between DP and HDP + spectral cut-off for \texttt{heat} and heavy-tailed noise.\label{tab:heat:p}}
	\setlength{\tabcolsep}{4pt}
	\begin{tabular}{c|ccc|ccc|ccc|}
		\toprule
		\multicolumn{1}{c}{}&\multicolumn{3}{c}{$D=2^7$} & \multicolumn{3}{c}{$D=2^9$} & \multicolumn{3}{c}{$D=2^{11}$} \\
		\cmidrule(l){2-4} \cmidrule(l){5-7} \cmidrule(l){8-10} 
		${\rm SNR}$ & $e_{\rm HDP}$  & $e_{\rm DP}$ & $e_{\rm opt}$ & $e_{\rm HDP}$  & $e_{\rm DP}$ & $e_{\rm opt}$ & $e_{\rm HDP}$  & $e_{\rm DP}$ & $e_{\rm opt}$\\
		1e-2 & 3.7e0  &  1e4    & 1e0   &  2.5e0    &  4e2    &  9.9e-1 & 1.4e0& 1.5e0& 9.7e-1\\
		1e-1 & 1e0   & 3.7e1     & 9.1e-1   & 9.2e-1    &  9.8e1    & 8.2e-1   & 8.2e-1& 4.3e0& 7.6e-1\\
		1e0 & 7.1e-1     & 4.2e0      & 6.5e-1  & 6.8e-1   &  8e0    & 5.2e-1 & 6.7e-1& 7.1e-1& 3.9e-1\\
		1e1 & 4.5e-1     & 5.3e-1      & 2.7e-1  & 2.8e-1   &  3.6e-1    & 2.1e-1 & 2.6e-1& 2.6e-1& 1.5e-1\\
		1e2  & 1.3e-1     & 4.2e-1      & 9.9e-2  & 1e-1   &  1.2e-1    & 7.4e-2 & 8.7e-2&9.6e-2 & 5.5e-2\\
		1e3 & 4.7e-2     & 8.6e-2      & 3.4e-2  & 4.2e-2   &  5.1e-2    & 2.5e-2 & 2.2e-2&3.1e-2 &1.9e-2  \\								
		\bottomrule
	\end{tabular}
\end{table}

}

\section{Concluding remarks}\label{sec5}

In this article we rigorously analysed a novel approach to solve statistical inverse problems without knowledge of the noise level, based on discretisation-adaptive choice of the regularisation parameter. The results are backed by numerical experiments. As future work it would be interesting to extend the rigorous analysis for spectral cut-off regularisation to more practical methods as the Landweber iteration or Tikhonov regularisation. { Here unreported numerical results  show the potential that the analysis can be extended combining the ideas from this manuscript and \cite{jahn2022discretisation}}. Another interesting idea would be to apply discretisation-adaptive regularisation to other heuristic methods as, e.g., the Hanke-Raus rule \cite{hanke1996general}, the quasi-optimality criterion \cite{tikhonov1963solution}, the L-curve criterion \cite{hansen1994regularization} or generalised cross-validation \cite{wahba1977practical}. From the numerical site it would be interesting to consider higher-dimensional integral equations.

\bibliographystyle{siamplain}
\bibliography{references}
\end{document}

%% file: hdp_shared.tex
\usepackage{lipsum}
\usepackage{amsfonts}
\usepackage{graphicx}
\usepackage{epstopdf}
\usepackage{algorithmic}
\usepackage{booktabs}
\ifpdf
  \DeclareGraphicsExtensions{.eps,.pdf,.png,.jpg}
\else
  \DeclareGraphicsExtensions{.eps}
\fi

\newsiamremark{remark}{Remark}
\newsiamremark{hypothesis}{Hypothesis}
\crefname{hypothesis}{Hypothesis}{Hypotheses}
\newsiamthm{claim}{Claim}
\newsiamremark{assumption}{Assumption}
\newsiamremark{example}{Example}

\headers{Noise level-free regularisation}{T. Jahn}

\title{Noise level free regularisation of general linear inverse problems under unconstrained white noise\thanks{
\funding{Funded  by  the  Deutsche  Forschungsgemeinschaft under Germany's Excellence Strategy - GZ 2047/1, Projekt-ID 390685813.}}}

\author{Tim Jahn\thanks{University of Bonn, Bonn, Germany 
  (\email{jahn@ins.uni-bonn.de}, \url{https://ins.uni-bonn.de/staff/jahn}).}
}

\usepackage{amsopn}

\newcommand{\R}{\mathbb{R}}
\newcommand{\N}{\mathbb{N}}

\newcommand{\E}{\mathbb{E}}

\newcommand{\omd}{\chi_{\Omega_{\delta}}}
\newcommand{\omdp}{\chi_{\Omega_{\delta}'}}
\newcommand{\omk}{\chi_{\Omega_\kappa}}
\DeclareMathOperator*{\argmin}{arg\!\min}
